\pdfoutput=1
\documentclass{amsart}

\usepackage{booktabs}
\usepackage{hyperref}
\hypersetup{
    colorlinks=true,       
    linkcolor=blue,          
    citecolor=blue,        
    filecolor=blue,      
    urlcolor=blue           
}
\usepackage[alphabetic,backrefs,lite]{amsrefs}
\usepackage{verbatim}

\usepackage[all]{xy}

\numberwithin{equation}{section}

\theoremstyle{plain}
\newtheorem{lemma}{Lemma}[section]
\newtheorem{theorem}[lemma]{Theorem}

\newtheorem{proposition}[lemma]{Proposition}

\theoremstyle{definition}

\theoremstyle{remark}

\newcommand{\cc}{\mathbb{C}}
\newcommand{\pp}{\mathbb{P}}
\newcommand{\rr}{\mathbb{R}}

\newcommand{\af}{\mathbb{A}}
\newcommand{\ff}{\mathbb{F}}

\newcommand{\zz}{\mathbb{Z}}
\newcommand{\ls}{\mathcal{L}}
\newcommand{\lsk}{\hat{\mathcal{L}}}
\newcommand{\lr}{\mathcal{R}}

\newcommand{\Osh}{{\mathcal O}}
\newcommand{\Ish}{{\mathcal I}}

\newcommand{\Xsh}{{\mathcal X}}

\newcommand{\Sym}{\operatorname{Sym}}
\newcommand{\rk}{\operatorname{rk}}

\begin{document}

\title[Secant varieties of  
Segre-Veronese
embeddings of $(\pp^1)^r$.]{Secant varieties of  
Segre-Veronese \\
embeddings of $(\pp^1)^r$.}
\author{Antonio Laface}
\address{
Departamento de Matem\'atica \newline
Universidad de Concepci\'on \newline 
Casilla 160-C \newline
Concepci\'on, Chile}
\email{antonio.laface@gmail.com}

\author{Elisa Postinghel}
\address{Centre of Mathematics for Applications, \newline University of Oslo \newline P.O. Box 1053 Blindern,\newline N0-0316 Oslo, Norway}
\email{elisa.postinghel@gmail.com}


\thanks{The first author was supported by Proyecto FONDECYT Regular 2011, N. 1110096. The second author was supported by Marie-Curie IT Network SAGA, [FP7/2007-2013] grant agreement PITN-GA-
2008-214584. Both authors were partially supported by Institut Mittag-Leffler.}

\begin{abstract}
We use a double degeneration technique to calculate 
the dimension of the secant variety of any Segre-Veronese
embedding of $(\pp^1)^r$.
\end{abstract}
\maketitle

\section*{Introduction}
The problem of determining the dimension 
of linear systems through double points 
in general position on an algebraic variety 
$X$ is a very hard one to be solved
in general. Complete results are known 
for varieties of small dimension~\cite{BaDr,ChCi,La,CaGeGi2,VT} 
and in any dimension just for $\pp^n$ by
the Alexander-Hirschowitz theorem~\cite{AlHi,Po}.

Secant varieties of Segre-Veronese varieties are not well-understood, so far, and many efforts have been made, see for example \cite{CaGeGi}, \cite{Ba}, \cite{Ab}, \cite{AbBr1}, \cite{AbBr2}.
In this paper we determine the dimension
of any linear system $\ls_{(d_1,\dots,d_r)}(2^n)$
of hypersurfaces of $(\pp^1)^r$ of multi-degree
$(d_1,\dots,d_r)$ through $n$ double points
in general position.
Solving this problem is equivalent, 
via Terracini's Lemma, to calculate the 
dimension of the secant variety of any
Segre-Veronese embedding $(\pp^1)^r\to \pp^N$,
defined by the complete linear system 
$|\Osh(d_1,\dots,d_r)|$ of $(\pp^1)^r$.
Our proof is based on a double induction,
on the dimension $r$ and on the degree
$d_1+\dots+d_r$. A basic step for our induction
is represented by the fundamental paper~\cite{CaGeGi3},
where the authors show that, if all the $d_i=1$,
then $\ls_{(1,\dots,1)}(2^n)$ has always but in 
one case ($r=4$) the expected dimension.

Our approach consists in degenerating $(\pp^1)^r$ to a union of two varieties both isomorphic to  $(\pp)^r$  and simultaneously degenerating the linear system to a linear system obtained as fibered product of  linear systems on the two components over the restricted system on their intersection. The limit linear system is somewhat easier than the original one, in particular this degeneration argument allows to use induction on the multi-degree and on $r$. This construction is a generalization of the technique introduced by Ciliberto and Miranda in \cite{CM1} and \cite{CM2} to study higher multiplicity interpolation problems in $\pp^2$ and recently generalized in \cite{Po} to the higher dimensional case to study linear systems of $d$-hypersurfaces of $\pp^r$ with a general collection of nodes.

The paper is organized as follows. In Section 1
we provide the basic notation. Section 2 contains
the statement of our main theorem together with 
its counterpart on secant varieties.
We enter the double degeneration technique 
in Section 3, while Section 4 is devoted to
apply this technique to compute the dimension 
of linear systems. In Section 5 we prove our 
theorem and finally we conclude in Section 6 
by describing all the linear systems whose
dimension is not the expected one.

\section{Notation}
Let $\ls:=\ls_{(d_1,\dots, d_r)}(2^n)$ be the 
linear system  of multi-degree 
$(d_1,\dots,d_r)$ hypersurfaces of $(\pp^1)^r$ 
which are singular at $n$ points in general 
position. Its \emph{virtual dimension} is defined 
to be
\[
 v(\ls):=\prod_{i=1}^r(d_i+1)-1-(r+1)n,
\]
i.e. the dimension of the linear system $|\mathcal{O}(d_1,\dots,d_{r})|$ of multi-degree $(d_1,\dots,d_r)$ hypersurfaces of $(\pp^1)^r$
minus the number of conditions imposed by the 
double points. The dimension of $\ls$ cannot 
be less than $-1$, hence we define the 
\emph{expected dimension} to be
\[
 e(\ls):=\max \{v(\ls), -1\}.
\]

If the conditions imposed by the assigned 
points are not linearly independent, the 
dimension of $\ls$ is greater that the 
expected one: in that case we say that 
$\ls$ is \emph{special}. Otherwise, if the 
dimension and the expected dimension of 
$\ls$ coincide, we say that $\ls$ is 
\emph{non-special}.

We are interested in investigating if a given 
linear system $\ls$ is non-special. 
The dimension of $\ls$ is \emph{upper-semicontinuous} 
in the position of the points in $(\pp^1)^n$ and
it achieves its minimum value when they are 
in \emph{general position}. 
Let $Z$ be the zero-dimensional scheme of 
length $(r+1)n$ given by $n$ double points 
in general position. With abuse of notation 
we will sometimes adopt the same symbol 
$\ls_{(d_1,\dots,d_r)}(2^n)$, or simply $\ls$, 
for denoting the linear system and the 
sheaf $\Osh(d_1,\dots,d_r)\otimes\Ish_Z$. 
With this in mind,
consider the following restriction 
exact sequence
\[
 \xymatrix@R=10pt{
 0\ar[r] & \ls\ar[r] &  \ls_{(d_1,\dots,d_r)}\ar[r] &{\ls_{(d_1,\dots,d_r)}}_{|Z}
 }.
\]
Taking cohomology, being 
$h^1((\pp^1)^r,\ls_{(d_1,\dots,d_r)})=0$ 
we get that $\ls$ is non-special if and 
only if 
\[
 h^0((\pp^1)^r,\ls)\cdot h^1((\pp^1)^r,\ls)=0.
\]

\section{The classification Theorem}

In this section we state our main theorem and
recall its connection with the dimension of the
secant varieties of the Segre-Veronese embeddings
of $(\pp^1)^r$.

\begin{theorem}\label{classif-p1r}
The linear system $\ls_{(d_1,\dots,d_r)}(2^n)$ 
of $(\pp^1)^r$ is non-special except in the 
following cases.
\begin{center}
\begin{tabular}{c|c|c|r|c}
$r$ & degrees & $n$ & $v(\ls)$ & $\dim(\ls)$\\
\midrule
$2$ & $(2,2a)$     & $2a+1$ & $-1$ & $0$\\
$3$ & $(1,1,2a)$  & $2a+1$ & $-1$ & $0$\\
       & $(2,2,2)$    & $7$        & $-2$ & $0$\\
$4$ & $(1,1,1,1)$ & $3$        & $0$  & $1$\\
\end{tabular}
\end{center}
\end{theorem}
All the exception in Theorem~\ref{classif-p1r}
where previously known. Moreover a proof of the
two dimensional case can be found in~\cite{La,VT}.
The three dimensional case was treated 
in~\cite{BaDr,CaGeGi2}. The special system in
dimension $4$ has been found in~\cite{CaGeGi}.

This theorem has an equivalent reformulation 
in terms of higher secant varieties of 
Segre-Veronese embeddings of products 
of $\pp^1$'s. 
Let $X$ be a projective variety of dimension 
$r$ embedded in $\pp^N$. 
The $n$-\emph{secant variety} $\textrm{S}_n(X)$ 
of $X$ is defined to be the Zariski closure 
of the union of all the linear spans in 
$\pp^N$ of $n$-tuples of independent points 
of $X$. We have, counting parameters, that
\[
 \dim(\textrm{S}_{n}(X))
 \leq
 \min\{nr+n-1,N\} 
\]
The variety $X$ is said to be 
$n$-\emph{defective} if strict inequality 
holds; it is said to be 
\emph{non}-$n$-\emph{defective} if equality holds.

Let $N:=\prod_{i=1}^r(d_i+1)-1$ and let 
$\nu: (\pp^1)^r\to\pp^N$ be the 
\emph{Segre-Veronese} embedding of multi-
degree $(d_1,\dots,d_r)$. 
Denote by $X$ the image of $\nu$.

\begin{theorem}\label{classif-p1r-secant}
Let $X$ be defined as above. 
The $n$-secant variety of $X$ is non-defective, 
with the same list of exceptions of 
Theorem \ref{classif-p1r}.
\end{theorem}

A hypersurface $S$ of $(\pp^1)^r$ of 
multi-degree $(d_1,\dots,d_r)$ corresponds 
via the Segre-Veronese embedding  to a 
hyperplane section $H$ of $X$. 
Moreover $S$ has a double point at $p$ 
if and only if $H$ is tangent to $X$ at 
$\nu(p)$. Now, fix $p_1,\dots,p_n$ 
general points in $(\pp^1)^r$ and consider 
the linear system $\ls_{(d_1,\dots,d_r)}(2^n)$ of multi-degree $(d_1,\dots,d_r)$ hyperfsurfaces singular at $p_1,\dots,p_n$. 
It corresponds to the linear system 
of hyperplanes $H$ in $\pp^N$ tangent 
to $X$ at $\nu(p_1),\dots,\nu(p_n)$.
This linear system has as base locus 
the general tangent space to $\textrm{S}_{n}(X)$. 
The following classical result, known as 
\emph{Terracini Lemma},  proves the 
equivalence between Theorem~\ref{classif-p1r} 
and Theorem~\ref{classif-p1r-secant}.
\begin{lemma}[Terracini's Lemma]
Let $X\subseteq \pp^N$ be an irreducible, 
non-degenerate, 
projective variety of dimension $r$. 
Let $p_1,\dots,p_n$ be general points 
of $X$, with $n\leq N+1$.  
Then 
\[
 T_{\textrm{S}_{n}(X),p}
 =
 <T_{X,p_1},\dots,T_{X,p_n}>,
\]
where $p\in<p_1\dots,p_n>$ is a general 
point in $\textrm{S}_{n}(X)$.
\end{lemma}

\section{The degeneration technique} 

We begin by constructing a toric flat 
degeneration of the variety $X$ into
a union of two toric varieties
$X_1$ and $X_2$ both isomorphic to $X$.

\subsection{A toric degeneration of $(\pp^1)^r$}
Let $P=P_{(d_1,\dots,d_r)}$ be the convex lattice polytope 
$[0,d_1]\times\cdots\times[0,d_r]\subseteq\rr^r$.
Its integer points define the toric map 
which is the Segre-Veronese embedding  
$\nu: (\pp^1)^r\to \pp^N$ given by the line bundle 
$\mathcal{O}(d_1,\dots,d_r)$. As before we will denote 
by $X$ the image of $\nu$.
Consider the function $\phi: P\cap\zz^r\to\zz$
defined by
\[
 \phi(v) =
 \begin{cases}
 0 & \text{ if } v_r\leq k\\
 v_r-k & \text{ if } v_r>k.
 \end{cases}
\]
It defines a {\em regular subdivision} 
of $P$ in the following way. Consider the
convex hull of the half lines 
$\{(v,t)\in P\times\rr_{\geq 0}: t\geq \phi(v)\}$.
This is an unbounded polyhedra with two lower 
faces. By projecting these faces on $P$
we obtain the  regular subdivision
\[
 \mathcal{T}
 =
 \{
 P_{(d_1,\dots,d_{r-1},d_r-k)},
 P_{(d_1,\dots,d_{r-1},k)}
 \},
\] 
with $P_{(d_1,\dots,d_{r-1},d_r-k)}
\cup P_{(d_1,\dots,d_{r-1},k)}=P$ and 
$P_{(d_1,\dots,d_{r-1},d_r-k)}\cap P_{(d_1,\dots,d_{r-1},k)}=P_{(d_1,\dots,d_{r-1})}$, 
where $P_{(d_1,\dots,d_{r-1})}=[0,d_1]
\times\cdots\times[0,d_{r-1}]$. 
We show the configuration of this toric 
degeneration in  Figure~\ref{pancarre} 
for $r=3$.
\begin{figure}
\[
\xymatrix@!0{
&&&*=0{}\ar@{--}[dd]\ar@{-}[rr]&&*=0{}\ar@{-}[rr]&*=0{}&*=0{}\ar@{-}[rr]&*=0{}\ar@{--}[dd]&*=0{}\ar@{-}[rr]&&*=0{}\ar@{-}[r]&*=0{}\ar@{-}[dd]\\
&&*=0{}\ar@{-}[ru]&&*=0{}&*=0{}&&\ar@{--}[ru]&&&&*=0{}\ar@{-}[ru]&\\
&*=0{}\ar@{-}[ru]^{\textrm{\normalsize{$d_1$}}}&&*=0{}\ar@{--}[rr]&&*=0{}\ar@{--}[rr]&*=0{}\ar@{--}[ru]&*=0{}\ar@{--}[rr]&&*=0{}\ar@{--}[r]&*=0{}\ar@{-}[ru]\ar@{--}[rr]&&*=0{}\\
*=0{}\ar@{-}[rr]\ar@{-}[ru]&&*=0{}\ar@{--}[ru]\ar@{-}[r]&*=0{}\ar@{-}[rr]&*=0{}&*=0{}\ar@{-}[rr]\ar@{--}[ru]&&*=0{}\ar@{-}[rr]\ar@{--}[ru]&&*=0{}\ar@{-}[ru]*=0{}\ar@{-}[ru]&&*=0{}\ar@{-}[ru]&\\
&*=0{}\ar@{--}[ru]&&*=0{}&&&*=0{}\ar@{--}[ru]&&&*=0{}\ar@{-}[u]&*=0{}\ar@{-}[ru]&&\\
*=0{}\ar@{--}[ru]\ar@{-}[uu]^{\textrm{\normalsize{$d_2$}}}\ar@{-}[r]&*=0{}\ar@{-}[r]&*=0{}\ar@{-}[r]_{\textrm{\normalsize{$d_3-k$}}}&*=0{}\ar@{-}[rr]&&*=0{}\ar@{--}[uu]\ar@{--}[ru]\ar@{-}[rr]&\ar@{-}[rr]_{\textrm{\normalsize{$k$}}}&*=0{}\ar@{-}[rr]&&*=0{}\ar@{-}[u]\ar@{-}[ru]&&&
}
\]
\caption{A regular subdivision of $P_{(d_1,d_2,d_3)}$}\label{pancarre}
\end{figure}
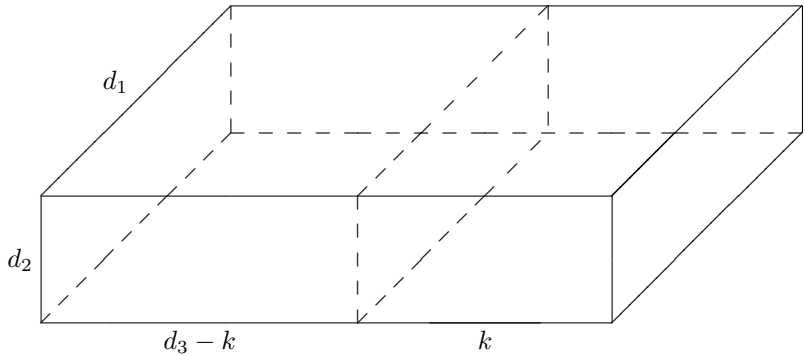
The regular subdivision defines a $1$-dimensional 
embedded degeneration in the following way.
Let $\Xsh$ be the toric subvariety of $\pp^N\times\af^1$
which is image of the toric morphism
\[
 (\cc^*)^r\times\cc^*\to \pp^N\times\af^1
 \qquad
 (x,t)\mapsto (\{t^{\phi(v)}x^v : v\in P\cap\zz^r\},t).
\]
Then $\Xsh$ admits two morphisms, induced by the projections,
onto $\pp^N$ and $\af^1$. Denote by 
$\pi:\Xsh\to\af^1$ the second morphism. 
The fiber $X_t$ of $\pi$ over $t$ is isomorphic
to $X$ if $t$ is not null and to the union
$X_1\cup X_2$ if $t=0$. Both the $X_i$'s are
isomorphic to $(\pp^1)^r$ and their intersection
$R$ is isomorphic to $(\pp^1)^{r-1}$. 
More precisely, $X_1$ is the Segre-Veronese 
embedding of $(\pp^1)^r$ defined by the complete 
linear system $|\mathcal{O}(d_1,\dots,d_{r-1},d_r-k)|$,
while $X_2$ is the Segre-Veronese embedding of 
$(\pp^1)^r$ defined by $|\mathcal{O}(d_1,\dots,d_{r-1},k)|$. 
The intersection $R$ is the Segre-Veronese 
embedding of $(\pp^1)^{r-1}$ given by 
$|\mathcal{O}(d_1,\dots,d_{r-1})|$.

\subsection{The $(k,n_2)$-degeneration of $\ls$}

A line bundle on $X_0$ corresponds to two line 
bundles, respectively on $X_1$ and on $X_2$, 
which agree on the intersection $R$.
We consider the linear system $\ls_t:=\ls$ 
 of multi-degree 
$(d_1,\dots,d_r)$ hypersurfaces of $X$ with $n$ assigned general 
points $p_{1,t},\dots,p_{n,t}$ of 
multiplicity $2$.

Fix a non-negative integer $n_1 \leq n$ 
and specialize $n_1$ points generically 
on $X_1$ and the other $n_2=n-n_1$ points 
generically on $X_2$, i.e. take a flat 
family $\{p_{1,t}\dots,p_{n,t}\}_{t\in\af^1}$ 
such that $p_{1,0},\dots,p_{n_1,0}\in X_1$ 
and $p_{n_1+1,0},\dots,p_{n_2,0}\in X_2$. 
The limiting linear system $\ls_0$ on $X_0$ 
is formed by the divisors in the flat limit 
of the bundle $\mathcal{O}(d_1,\dots,d_r)$ 
on the general fiber $X_t$, singular at 
$p_{1,0},\dots,p_{n,0}$. 
Consider the following linear systems:
\begin{equation}
 \begin{array}{llllll}
 \ls_1 & := & \ls_{(d_1,\dots,d_{r-1},d_r-k)}(2^{n_1})
 &
 \ls_2 & := & \ls_{(d_1,\dots,d_{r-1},k)}(2^{n_2})\\[3pt]
 \lsk_1 & := & \ls_{(d_1,\dots,d_{r-1},d_r-k-1)}(2^{n_1})
 &
 \lsk_2 & := & \ls_{(d_1,\dots,d_{r-1},k-1)}(2^{n_2}),
 \end{array}
\end{equation}
where $\ls_i$, $\lsk_i$ are defined on $X_i$
and $\lsk_i$ is the kernel of the restriction
of $\ls_i$ to $R$. This is given by the 
exact sequence:
\[
 \xymatrix@R=10pt{
 0\ar[r] & \lsk_i\ar[r] & \ls_i\ar[r] 
 & {\ls_i}_{|R}=:\lr_i\ar[r] & 0\\
 }
\]
The kernel $\lsk_i$ consists of those divisors of 
$\ls_i$ which vanish identically on $R$, 
i.e. the divisors in $\ls_i$ containing 
$R$ as component.
An element of $\ls_0$ consists either of a 
divisor on $X_1$ and a divisor  on $X_2$, 
both satisfying the conditions imposed by 
the multiple points, which restrict to the 
same divisor on $R$, or it is a divisor 
corresponding to a section of the bundle 
which is identically zero on $X_1$ (or on $X_2$) 
and which gives a general divisor in $\ls_2$ 
(or in $\ls_1$ respectively) containing $R$ 
as a component.
We have, by upper-semicontinuity, that
$\dim(\ls_0) \geq \dim(\ls)$.
\begin{lemma}\label{upp-semicont}
In the above notation, if $\dim(\ls_0) = e(\ls)$, 
then the linear system $\ls$ has the expected 
dimension, i.e. it is non-special.
\end{lemma}

We will say that $\ls_0$ is obtained from $\ls$ 
by a $(k,n_2)$-\emph{degeneration}.

\subsection{The $(k,n_2,\beta)$-degeneration of $\ls$}\label{IIdeg}

Fix a non-negative integer $\beta < \min(r,n_2)$.
Suppose that we have already performed 
a $(k,n_2)$-degeneration of $\ls$.
We perform a further degeneration of 
the linear system $\ls_0$ on the central fiber 
by sending $\beta$ among the $n_2$ double 
points of $X_2$ to the intersection $R$ 
of $X_1$ and $X_2$. 
As a result we obtain the following systems
\begin{equation}
 \begin{array}{llllll}
 \ls_1 & := & \ls_{(d_1,\dots,d_{r-1},d_r-k)}(2^{n_1})
 &
 \ls_2 & := & \ls_{(d_1,\dots,d_{r-1},k)}(2^{n_2})\\[3pt]
 \lsk_1 & := & \ls_{(d_1,\dots,d_{r-1},d_r-k-1)}(2^{n_1})
 &
 \lsk_2 & := & \ls_{(d_1,\dots,d_{r-1},k-1)}(2^{n_2-\beta},1^\beta),
 \end{array}
\end{equation}
On the intersection $R$ we have
\[ 
 \lr_2 \subseteq \ls_{(d_1,\dots,d_{r-1})}(2^\beta).
\]
Observe that we are abusing our original 
notation for $\ls_2$ and $\lsk_2$ since
the $n_2$ double points on $X_2$ are no
longer in general position.
We will say that $\ls_0$ is obtained 
from $\ls$ by a $(k,n_2,\beta)$-\emph{degeneration}, 
implying that if $\beta>0$ we perform 
the double degeneration, while if $\beta=0$ 
we do not need to perform it.

\section{Computing the dimension of the limit system}

Our aim is to compute $\dim(\ls_0)$ by recursion. 
The simplest cases occurs when all the 
divisors in $\ls_0$ come from a section 
which is identically zero on one of the 
two components: in those cases the matching 
sections of the other system must lie in 
the kernel of the restriction map. 
\begin{lemma}\label{simple deg}
If $\ls_2$ is empty, then $\dim(\ls_0)=\dim(\lsk_1)$.
\end{lemma}

If, on the contrary, the divisors on $\ls_0$ 
consist of a divisor on $X_1$ and a divisor 
on $X_2$, both not identically zero, which 
match on $R$, then the dimension of $\ls_0$ 
depends on the dimension of the intersection 
$\lr:=\lr_1\cap\lr_2$ of the restricted 
systems. 
\begin{lemma}\label{formula l_0} 
$\dim(\ls_0)=\dim(\lr)+\dim(\lsk_1)+\dim(\lsk_2)+2.$
\end{lemma}
\begin{proof}
A section of $H^0(X_0,\ls_0)$ is obtained 
by taking an element in $H^0(R,\lr)$ and 
choosing preimages of such an element: 
$h^0(X_0,\ls_0)=h^0(R,\lr)+h^0(X_1,\lsk_1)
+h^0(X_2,\lsk_2)$. 
Thus, at the linear system level we get 
the formula.
\end{proof}

\subsection{Transversality of the restricted systems}

The crucial point is to compute the 
dimension of $\lr$. If the systems 
$\lr_1,\lr_2 \subseteq 
|\mathcal{O}_{R}(d_1,\dots,d_{r-1})|$ 
are \emph{transversal}, i.e. if they 
intersect properly, then the dimension 
of the intersection $\lr$ is easily 
computed. 
It is immediate to see that if 
$\lsk_1$ (or $\lsk_2$) is non-special with virtual dimension 
$\geq -1$, and moreover $\ls_1$ 
(resp. $\ls_2$) is non-special, then 
the restricted system $\lr_1$ 
(resp. $\lr_2$) is the complete 
linear system 
$|\mathcal{O}_R(d_1,\dots,d_{r-1})|$ 
and transversality trivially holds. 
If none of the two kernel systems 
satisfies this property, so that 
$\lr_1,\lr_2$ are both proper subsets 
of $|\mathcal{O}_R(d_1,\dots,d_{r-1})|$, 
then we need to prove that they 
intersect properly. 

\subsection{Applying the double degeneration technique}
In what follows we will make use of
just two types of degenerations, so
we consider them in detail now.
Choose
\begin{equation}\label{case1}
 k=r
 \qquad
 n_2=\prod_{i=1}^{r-1}(d_i+1).
\end{equation}
Then $v(\ls_2)=-1$. Since transversality 
trivially holds, and in particular 
$\lr_1\cap\lr_2=\emptyset$, the problem 
of studying $\ls$ is recursively translated 
to the problem of studying the lower degree 
linear system 
$\lsk_1=\ls_{(d_1,\dots,d_{r-1},d_r-r-1)}(2^{n_1})$ 
which has the same virtual dimension. 
This method is simple and useful if one 
wants to focus on some particular $r$ and 
apply induction on the multi-degree 
$(d_1,\dots,d_r)$, provided that the 
base steps of the induction, i.e. any case 
$(d_1,\dots,d_r)$ with 
$d_1\leq\cdots\leq d_r\leq r+1$ are analysed 
in advance.   

However, the aim of this paper is to cover 
all cases of linear systems of divisors of 
any multi-degree $(\pp^1)^r$, for any $r$. 
To this end, we want to exploit induction 
not only on the multi-degree but also on 
$r$ in order to have a more compact and 
powerful method.  This can be done by choosing 
$k=1$ and applying a double degeneration 
as described in Section \ref{IIdeg}. 
This argument consists in an ad hoc 
adaptation of the degeneration technique 
implied in \cite{Po} to prove the 
non-speciality of linear systems of 
divisors of degree $d$ in $\pp^r$ with 
a general collection of double points.

Choose integers $k,n_2,\beta$ 
as follows
\begin{equation}\label{case2}
 k=1
 \qquad
 \prod_{i=1}^{r-1}(d_i+1) = r(n_2-\beta)+\beta,
 \qquad
 \beta\in\{0,\dots,r-1\}
\end{equation}
and perform the double degeneration of $X$, 
$\ls$ described above. 
Observe that elements of $\lsk_2$ are in 
bijection with elements of 
$\ls_{(d_1,\dots,d_{r-1})}(2^{n_2-\beta},1^\beta)$, 
that is the linear system 
of hypersurfaces of $|\Osh(d_1,\dots,d_{r-1})|$ 
on $(\pp^1)^{r-1}$ singular at $n_2-\beta$ points 
and passing through $\beta$ points, all of them in general position:
\[
 \lsk_2
 =
 \ls_{(d_1,\dots,d_{r-1},0)}(2^{n_2-\beta},1^\beta)
 \cong
 \ls_{(d_1,\dots,d_{r-1})}(2^{n_2-\beta},1^\beta).
\]
Observe that the last system has
virtual dimension $-1$ by the definition
of $n_2$ and $\beta$.
Let $\pi_r:(\pp^1)^r\to(\pp^1)^{r-1}$
be the projection on the first $r-1$
factors. Observe that 
\[
 \lr_2\subseteq
 \ls_{(d_1,\dots,d_{r-1})}(1^{n_2-\beta},2^\beta)
\]
since elements of $\ls_2$ contain the lines 
$\pi_r^{-1}(p_i)$ through each one of the $n_2$ 
double points $p_i$ of $X_2$. 
\begin{lemma}\label{res-lemma}
Let $\lr_2$ be as above. If 
$\ls_{(d_1,\dots,d_{r-1})}(2^{n_2})$ 
is non-special, then 
$\lr_2=\ls_{(d_1,\dots,d_{r-1})}(1^{n_2-\beta},2^\beta)$.
\end{lemma}
\begin{proof}
Since $\ls_{(d_1,\dots,d_{r-1})}(2^{n_2})$
is non-special, then 
$\ls_{(d_1,\dots,d_{r-1})}(2^{n_2},1^\beta)$
is non-special due to the fact
that the $\beta$ simple points 
are in general position on $R$.
Consider the exact sequence of 
linear systems:
\[
 \xymatrix@R=10pt{
 0\ar[r] & \lsk_2\ar[r] & \ls_2\ar[r] 
 & \lr_2\ar[r] & 0.\\
 }
\]
Since $\ls_{(d_1,\dots,d_{r-1})}(2^{n_2},1^\beta)$ 
is non-special of virtual dimension $-1$,
then it is empty. Thus $\lsk_2$ is empty
so that $\lr_2$ is the complete
linear system 
$\ls_{(d_1,\dots,d_{r-1})}(1^{n_2-\beta},2^\beta)$
obtained by restricting $\ls_2$
to $R$.
\end{proof}

In order to prove that transversality holds, it is enough to prove that $\beta$ double points 
supported on $R$ impose independent conditions 
to the linear system $\lr_1$. 
\begin{lemma}[Transversality Lemma]\label{transversality lemma}
In the same notation as above, if the linear system $\ls_{(d_1,\dots,d_r-k)}(2^{n_1+\beta})$ is non-special, for points in general position, then $\lr_1$ and $\lr_2$ intersect transversally on $R$. In particular
\[
 \dim(\lr)
 = 
 \max\{\dim(\lr_1)-(n_2-\beta)-r\beta,-1\}.
\]
\end{lemma}
\begin{proof}
Since $\beta<r$, the scheme formed by $n_1$ double points on $X_1$ and $\beta$ additional double points on $R$ is general in $X_1$. If $\beta$ nodes in general position impose independent conditions to $\ls_1=\ls_{(d_1,\dots,d_r-k)}(2^{n_1})$, namely if $\ls_{(d_1,\dots,d_r-k)}(2^{n_1+\beta})$ is non-special, then the $\beta$ nodes supported on $R$ give independent conditions to $\lr_1$. Moreover, the intersection $\lr$ is formed 
by those elements in $\lr_1$ that are singular 
at $\beta$ points and pass trough $n_2-\beta$ 
points in general position. This proves the formula for $\dim(\lr)$.
\end{proof}

\section{The proof of the Classification Theorem}

Aim of this section is to prove 
Theorem~\ref{classif-p1r} by induction
on the multi-degree and on the 
dimension $r$ of the variety.

Define the integers
\[
n^-:=\left\lfloor \frac{1}{r+1}\prod_{i=1}^r(d_i+1)\right\rfloor, \qquad 
n^+:=\left\lceil \frac{1}{r+1}\prod_{i=1}^r(d_i+1)\right\rceil.
\]
Notice that if non-speciality holds for a collection of $n^-$ double points, then it holds for a smaller number of double points. On the other hand, if there are no hypersurfaces of multi-degree $(d_1,\dots,d_r)$ with $n^+$ general nodes, the same is true adding other nodes. It is enough to analyse the cases $n^-\leq n \leq n^+$.

In what follows we will make use of
the following fact proved in~\cite{CaGeGi}:
\begin{equation}\label{toPn}
 \dim \ls_{(d_1,\dots,d_r)}(2^n)
 =
 \dim \ls_d(d-d_1,\dots,d-d_r,2^n),
\end{equation}
where $d=d_1+\dots+d_r$ and the
elements of the right hand side system 
are hypersurfaces of $\pp^r$ with
$r$ points of multiplicity $d-d_1,\dots,d-d_r$
and $n$ double points, all in general position.

\begin{proposition}\label{base steps}
If $d_1,\dots,d_r$ are positive integers
as in the following table, then 
$\ls_{(d_1,\dots,d_r)}(2^n)$
is non-special for any value of $n$.
\begin{center}
\begin{tabular}{c|r|r}
$r$ & degrees & bound \\
\midrule
$2$ & $(d_1,d_2)$     & $d_1,d_2\leq 6$ \\
$3$ & $(d_1,d_2,d_3)$  & $d_1,d_2,d_3\leq 6$ \\
$4$ & $(d_1,d_2,d_3,d_4)$ & $d_1,\dots,d_4\leq 4$ \\
    & $(1,1,d_3,d_4)$ & $d_3,d_4\leq 6$ \\
    & $(2,2,2,5)$ & \\
$5$ & $(1,1,1,1,d_5)$ & $d_5\leq 5$\\
    
\end{tabular}
\end{center}
\end{proposition}
\begin{proof}
By applying~\eqref{toPn} we reduce to 
consider the linear system of $\pp^r$ 
given by $\ls_d(d-d_1,\dots,d-d_r,2^n)$.
By means of the programs at this url:
\url{http://www2.udec.cl/~alaface/software/p1n.html}
we analyze these systems by extracting $n+r$ 
random points of $\pp^r$ on the finite field
$\ff_{307}$ and calculate the degree $d$
part of the ideal with the assigned fat points.
\end{proof}

\subsection{The case $r=2$}
This case is well known in literature,
see for example~\cite[Proposition 5.2]{La}
or~\cite{VT}.
In any case, for the sake of completeness, 
we provide a complete proof of this case 
as well.

\begin{proposition}
Let $d_1,d_2$ be positive integers.
Then $\ls_{(d_1,d_2)}(2^n)$
is special only in the cases described 
by Proposition~\ref{special}.
\end{proposition}
\begin{proof}
By Proposition~\ref{base steps} 
it is enough to concentrate 
on the case $d_3\geq 6$.
Apply a $(2,d_1+1)$-degeneration.
Then $v(\ls_2)=-1$ and $v(\lsk_1)=v(\ls)$.
By induction hypothesis $\ls_2$ is non-special 
so that it is empty and 
$\dim(\ls_0)=\dim(\lsk_1)$
by Lemma~\ref{simple deg}.
If $(d_1,d_2)$ is not equal
to $(1,2a)$, then then $\lsk_1$ 
is non-special, by induction 
so we conclude by Lemma~\ref{upp-semicont}.
\end{proof}

\subsection{The case $r=3$}
We will prove the non-speciality of $\ls_{(d_1,d_2,d_3)}(2^{n})$, 
with all the the $d_i$ positive and 
$(d_1,d_2,d_3,n)\neq(2,2,2,7),(1,1,2a,2a+1)$.
This case also was previously known,
see~\cite{CaGeGi2} or~\cite{BaDr}.

\begin{proposition}
Let $d_1,d_2,d_3$ be positive integers.
Then $\ls_{(d_1,d_2,d_3)}(2^n)$
is special only in the cases described 
by Proposition~\ref{special}.
\end{proposition}
\begin{proof}
By Proposition~\ref{base steps} 
it is enough to concentrate 
on the case $d_3\geq 6$.
Apply a $(3,n_2)$-degeneration with 
$n_2:=(d_1+1)(d_2+1)$.
Then $v(\ls_2)=-1$ and $v(\lsk_1)=v(\ls)$.
Observe that by induction hypothesis
$\ls_2$ is non-special so that
it is empty and $\dim(\ls_0)=\dim(\lsk_1)$
by Lemma~\ref{simple deg}.
If $(d_1,d_2,d_3)$ is not equal
to $(1,1,2a)$, then then $\lsk_1$ 
is non-special, by induction 
so we conclude by Lemma~\ref{upp-semicont}.
\end{proof}

\subsection{The case $r=4$} We proceed with our 
investigation by proving the non-speciality
of $\ls_{(d_1,\dots,d_4)}(2^n)$ for
$(d_1,d_2,d_3,d_4)$ distinct from $(1,1,1,1)$.

\begin{proposition}
Let $d_1,d_2,d_3,d_4$ be positive integers.
Then $\ls_{(d_1,\dots,d_4)}(2^n)$
is special only in the cases described 
by Proposition~\ref{special}.
\end{proposition}
\begin{proof}
We begin by analyzing four distinct cases.
\begin{enumerate}
\item 
$(d_1,d_2,d_3,d_4)=(1,1,1,d_4)$
with $d_4\geq7$. In this case we perform a 
$(4,8)$-degeneration obtaining the systems
\[
 \ls_1 = \ls_{(1,1,1,d_4-4)}(2^{n-8})
 \qquad
 \ls_2 = \ls_{(1,1,1,4)}(2^8).
\]
Since $\ls_2$ is empty by 
Proposition~\ref{base steps} and 
$\lsk_1=\ls_{(1,1,1,d_4-5)}(2^{n-8})$ 
is non-special by induction on $d_4$, 
we conclude using 
Lemma~\ref{simple deg}.
\item 
$(d_1,d_2,d_3,d_4)=(1,1,4,d_4)$, 
with $d_4\geq 6$. In this case we perform a
$(4,20)$-degeneration obtaining the systems
\[
 \ls_1 = \ls_{(1,1,4,d_4-4)}(2^{n-20})
 \qquad
 \ls_2 = \ls_{(1,1,4,4)}(2^{20}).
\]
Since $\ls_2$ is empty by 
Proposition~\ref{base steps} and 
$\lsk_1=\ls_{(1,1,4,d_4-5)}(2^{n-20})$ 
is non-special by induction on $d_4$, 
we conclude using 
Lemma~\ref{simple deg}.
\item
$(d_1,d_2,d_3,d_4)=(1,1,d_3,d_4)$, 
with $d_3,d_4\geq 6$. In this case we perform a
$(4,4d_4+4)$-degeneration obtaining the systems
\[
 \ls_1 = \ls_{(1,1,d_3,d_4-4)}(2^{n-4d_4-4})
 \qquad
 \ls_2 = \ls_{(1,1,d_3,4)}(2^{4d_4+4}).
\]
Since $\ls_2$ is empty by 
Proposition~\ref{base steps} and 
$\lsk_1=\ls_{(1,1,d_3,d_4-5)}(2^{n-4d_4-4})$
is non-special by induction on $d_4$, 
we conclude using 
Lemma~\ref{simple deg}.

\item
$(d_1,d_2,d_3,d_4)=(2,2,2,d_4)$, 
with $d_4\geq6$. In this case we perform a
$(4,27)$-degeneration obtaining the systems
\[
 \ls_1 = \ls_{(2,2,2,d_4-4)}(2^{n-27})
 \qquad
 \ls_2 = \ls_{(2,2,2,4)}(2^{27}).
\]
Since $\ls_2$ is empty by 
Proposition~\ref{base steps} and 
$\lsk_1=\ls_{(2,2,2,d_4-5)}(2^{n-27})$
is non-special by induction on $d_4$, 
we conclude using 
Lemma~\ref{simple deg}.
\end{enumerate}

Suppose now that $(d_1,d_2,d_3,d_4)$ is 
distinct from $(1,1,1,d_4)$, $(2,2,2,d_4)$
and $(1,1,2a,d_4)$. We perform a 
$(1,n_2,\beta)$-degeneration, with $n_2$ 
and $\beta$ defined as in~\eqref{case2}
obtaining the systems
\[
 \ls_1 = \ls_{(d_1,d_2,d_3,d_4-1)}(2^{n_1})
 \qquad
 \ls_2 = \ls_{(d_1,d_2,d_3,1)}(2^{n_2}). 
\]
Recall that exactly $\beta$ of the $n_2$
double points of $X_2$ are sent to $R$.
Thus the kernels are:
\[
 \lsk_1 = \ls_{(d_1,d_2,d_3,d_4-2)}(2^{n_1})
 \qquad
 \lsk_2 = \ls_{(d_1,d_2,d_3,0)}(2^{n_2-\beta},1^\beta). 
\]
Observe that $\lsk_2\cong
\ls_{(d_1,d_2,d_3)}(2^{n_2-\beta},1^\beta)$
and the last system is empty by induction
since it has virtual dimension $-1$.
Also $\lsk_1$ is empty since it is
non-special by induction and has 
negative virtual dimension.
Thus $\ls$ is non-special by 
the assumption on the $d_i$ and
by Lemma~\ref{formula l_0} and Lemma~\ref{transversality lemma}.
\end{proof}

\subsection{The case $r=5$}
We show that there are no special
systems in dimension $5$.

\begin{proposition}
Let $d_1,\dots,d_5$ be positive integers.
Then $\ls_{(d_1,\dots,d_5)}(2^n)$ is
non-special.
\end{proposition}
\begin{proof}
If $(d_1,d_2,d_3,d_4,d_5)=(1,1,1,1,d_5)$,
with $d_5\geq 6$, we perform a $(5,16)$-degeneration 
obtaining the linear systems
\[
 \ls_1 = \ls_{(1,1,1,1,d_5-5)}(2^{n-16})
 \qquad
 \ls_2 = \ls_{(1,1,1,1,5)}(2^{16}). 
\]
Since $\ls_2$ is empty by 
Proposition~\ref{base steps} and 
$\lsk_1=\ls_{(1,1,1,1,d_5-6)}(2^{n-16})$
is non-special by induction, we conclude using 
Lemma~\ref{simple deg}.

Suppose now $(d_1,d_2,d_3,d_4,d_5)\neq(1,1,1,1,d_5)$.
Like in the four dimensional case
we perform a $(1,n_2,\beta)$-degeneration, 
with $n_2$ and $\beta$ defined as in~\eqref{case2}
obtaining the systems
\[
 \ls_1 = \ls_{(d_1,\dots,d_4,d_5-1)}(2^{n_1})
 \qquad
 \ls_2 = \ls_{(d_1,\dots,d_4,1)}(2^{n_2}). 
\]
Recall that exactly $\beta$ of the $n_2$
double points of $X_2$ are sent to $R$.
Thus the kernels are:
\[
 \lsk_1 = \ls_{(d_1,\dots,d_4,d_5-2)}(2^{n_1})
 \qquad
 \lsk_2 = \ls_{(d_1,\dots,d_4,0)}(2^{n_2-\beta},1^\beta). 
\]
Observe that $\lsk_2\cong
\ls_{(d_1,d_2,d_3,d_4)}(2^{n_2-\beta},1^\beta)$
and the last system is empty by induction
since it has virtual dimension $-1$.
Also $\lsk_1$ is empty since it is
non-special by induction and has 
negative virtual dimension.
Thus $\ls$ is non-special by 
the assumption on the $d_i$ and
by Lemma~\ref{formula l_0}.
\end{proof}

\subsection{The cases $r\geq 6$}
We show that there are no special
systems in dimension $\geq 6$.

\begin{proposition}
Let $r\geq6$. Let $d_1,\dots,d_r$ be positive integers. 
Then the linear system $\ls_{(d_1,\dots,d_r)}(2^n)$ is
non-special.
\end{proposition}
\begin{proof}
We proceed by induction on $d:=d_1+\dots+d_r$.
If $d=r$, then $d_i=1$ for any $i$. This case
is non-special as proved in~\cite{CaGeGi3}.
If $d>r$, then after possibly reordering
the $d_i$ we can assume that $d_r>1$.
We perform a $(1,n_2,\beta)$-degeneration, 
with $n_2$ and $\beta$ defined as in~\eqref{case2}
obtaining the systems
\[
 \ls_1 = \ls_{(d_1,\dots,d_{r-1},d_r-1)}(2^{n_1})
 \qquad
 \ls_2 = \ls_{(d_1,\dots,d_{r-1},1)}(2^{n_2}). 
\]
Recall that exactly $\beta$ of the $n_2$
double points of $X_2$ are sent to $R$.
Thus the kernels are:
\[
 \lsk_1 = \ls_{(d_1,\dots,d_{r-1},d_r-2)}(2^{n_1})
 \qquad
 \lsk_2 = \ls_{(d_1,\dots,d_{r-1},0)}(2^{n_2-\beta},1^\beta). 
\]
Observe that $\lsk_2\cong
\ls_{(d_1,\dots,d_{r-1})}(2^{n_2-\beta},1^\beta)$
and the last system is empty by induction
since it has virtual dimension $-1$.
The kernel system $\lsk_1$ has negative 
virtual dimension for any $(d_1,\dots,d_r)$, as one can easily check. 
Furthermore the dimension of the limit 
system equals the dimension of the intersection 
$\lr_1\cap\lr_2$ by Lemma~\ref{formula l_0}. 
Indeed, since $\lr_1$ and $\lr_2$ intersect 
transversally by Lemma~\ref{transversality lemma}, 
we have
\begin{align*}
 \dim (\lr_1\cap\lr_2) &=\max\{\dim(\lr_1)-(n_2-\beta)-r\beta,-1\}\\
 & =\max\{\prod_{i=1}^r(d_i+1)-1-(r+1)n,-1\}=e(\ls).
\end{align*}
\end{proof}

\section{Special linear systems}
In this section we complete the proof of 
Theorem~\ref{classif-p1r} by calculating 
the dimension of each special system.

\begin{proposition}~\label{special}
The following linear systems have the 
stated virtual and effective dimensions.
In each case the system $\ls$ is singular
along a smooth rational variety.
\begin{center}
\begin{tabular}{c|c|c|r|c}
$r$ & degrees & $n$ & $v(\ls)$ & $\dim(\ls)$\\
\midrule
$2$ & $(2,2a)$     & $2a+1$ & $-1$ & $0$\\
$3$ & $(1,1,2a)$  & $2a+1$ & $-1$ & $0$\\
       & $(2,2,2)$    & $7$        & $-2$ & $0$\\
$4$ & $(1,1,1,1)$ & $3$        & $0$  & $1$\\
\end{tabular}
\end{center}
\end{proposition}
\begin{proof}
Observe that Proposition~\ref{base steps}
already provide a computer based proof
of the first part of the statement. 
Anyway we prefer to give an alternative proof 
also for the calculation of the dimension
in order to make things more explicit.
The dimension of each such system can be found
by repeated use of formula~\eqref{toPn} together 
with the following one (see~\cite{LaUg}). Let
$\ls_{\pp^n}:=\ls_d(m_1,\dots,m_{n+1},m_{n+2},\dots,m_r)$,
then
\begin{equation}\label{standard}
 \dim\ls_{\pp^n}
 =
 \dim\ls_{d+k}(m_1+k,\dots,m_{n+1}+k,m_{n+2},\dots,m_r),
\end{equation}
where $k=(n-1)d-m_1-\dots-m_{n+1}$.
Our analysis begins with the two dimensional case.
If $r=2$ by applying~\eqref{toPn} first and~\eqref{standard} after 
we get
\begin{align*}
 \dim \ls_{(2,2a)}(2^{2a+1})
 & = 
 \dim \ls_{2a+2}(2a,2^{2a+2})\\
 & = 
 \dim \ls_{2a}(2(a-1),2^{2a})\\
 & = 
 \dim \ls_{2}(2^{2}) = 0.
\end{align*}
The linear system is singular along the unique
smooth rational curve in $\ls_{(1,a)}(1^{2a+1})$.
If $r=3$ we have to analyze two cases. The first
one is
\begin{align*}
 \dim \ls_{(1,1,2a)}(2^{2a+1})
 & = 
 \dim \ls_{2a+2}((2a+1)^2,2^{2a+2})\\
 & = 
 \dim \ls_{2a}(2a-1,2^{2a})\\
 & = 
 \dim \ls_{2}(2^2,1^2) = 0.
\end{align*}
Observe that the dimension of the last
system is $0$ since it contains just
a pair of planes which intersect along a line 
through the first two points. These two planes
can be seen also in the original system.
They are $D_1+D_2\in \ls_{(1,1,2a)}(2^{2a+1})$, 
with $D_1\in\ls_{(1,0,a)}(1^{2a+1})$ and 
$D_2\in\ls_{(0,1,a)}(1^{2a+1})$.
Thus the system is singular along the curve 
$D_1\cap D_2$ which is smooth by Bertini's 
theorem and rational by the adjunction formula.
The second case in dimension $r=3$ is
\begin{align*}
 \dim \ls_{(2,2,2)}(2^7)
 & = 
 \dim \ls_{6}(4^3,2^7)\\
 & = 
 \dim \ls_{4}(2^9) = 0.
\end{align*}
This last system is well known from
Alexander-Hirschowitz theorem, see~\cite{AlHi,Po}.
Observe that $\ls_{(2,2,2)}(2^7)$ is singular
along the smooth rational surface defined by 
$\ls_{(1,1,1)}(1^7)$.
Our last case is when $r=4$. Here we have
\begin{align*}
 \dim \ls_{(1,1,1,1)}(2^3)
 & = 
 \dim \ls_{4}(3^4,2^3)\\
 & = 
 \dim \ls_{2}(2^2,1^4)\\
 & = 
 \dim \ls_{1}(1^3) = 1.
\end{align*}
Consider now the linear system 
$\ls_{(1,1,1,1)}(2^3)$, where the three 
double points have coordinates:
$p_1:=([1:0],[1:0],[1:0],[1:0])$,
$p_2:=([0:1],[0:1],[0:1],[0:1])$,
$p_3:=([1:1],[1:1],[1:1],[1:1])$.
An easy calculation shows that 
the elements of the linear system
are zero locus of the sections
of the pencil:
\[
 \alpha_0\,(x_0y_1-x_1y_0)(z_0w_1-z_1w_0)
 +
 \alpha_1\,(x_0z_1-x_1z_0)(y_0w_1-y_1w_0).
\]
In particular the general element
of the pencil is singular along the curve
$V(x_0y_1-x_1y_0)\cap V(z_0w_1-z_1w_0)\cap
V(x_0z_1-x_1z_0)\cap V(y_0w_1-y_1w_0)$.
This is exactly the image of the
diagonal morphism 
$\Delta:\pp^1\to\pp^1\times\pp^1
\times\pp^1\times\pp^1$,
hence it is a smooth rational curve.
\end{proof}

\subsection{The $7$-secant variety of 
the $(2,2,2)$-embedding of $(\pp^1)^3$.}
Our list of special systems contains the 
case $\ls_{(2,2,2)}(2^7)$ which 
has virtual dimension $-2$ and dimension 
$0$. In this section we will show that 
the corresponding secant variety enjoys 
a symmetry which allows us to determine 
its equation.

Let $\varphi:(\pp^1)^3\to\pp^{26}$ be 
the Segre-Veronese embedding of $(\pp^1)^3$
defined by the complete linear system 
$\ls_{(2,2,2)}$, and let $S$ be the 
$7$-secant variety to 
the image of $\varphi$. Denote, by abuse 
of notation, with the same letter the 
corresponding linear map of vector spaces 
\[
 \xymatrix@R=0pt{
 \varphi: V\otimes V\otimes V\ar[r]
 &
 \Sym^2(V)\otimes\Sym^2(V)\otimes\Sym^2(V)\\
 \hspace{8mm}a\otimes b\otimes c\ar@{|->}[r] & 
 a^2\otimes b^2\otimes c^2,\hspace{30mm}
 }
\]
where each copy of $V$ represents the 
vector space of degree $1$ homogeneous 
polynomials in two variables. Let $V^*$ 
be the dual of $V$. 
Let $W:=\Sym^2(V)^{\otimes^3}$ be the 
codomain of $\varphi$. Given a polynomial 
$f\in W$, it belongs to the $7$-secant 
variety $S$ if and only if
\[
 f = a_1^2b_1^2c_1^2+\dots+a_7^2b_7^2c_7^2,
\]
where we are adopting the short notation 
$abc$ for $a\otimes b\otimes c$.
Consider the catalecticant map
\[
 C_f: (V^*)^{\otimes^3}
 \to
 V^{\otimes^3}
 \qquad
 a^*b^*c^*\mapsto a^*b^*c^*(f).
\]
We have that if $f\in S$, then 
$\rk(C_f)\leq 7$. Hence, as $f$ varies 
in $W$, the determinant of $C_f$ gives 
an equation that vanishes on $S$. 
Since this equation turns out to be 
irreducible and $S$ is an irreducible 
hypersurface, this will provide us the 
equation of $S$. In coordinates, let 
$a_ib_jc_k$, with $i,j,k\in\{0,1\}$ 
be a basis of $V^{\otimes^3}$ and let 
$a_i^*b_j^*c_k^*$ be a basis of the 
dual vector space $(V^*)^{\otimes^3}$ 
such that $a_i^*b_j^*c_k^*(a_ib_jc_k)=1$. 
An element $f\in W$ can be uniquely 
written as
\[
 f = z_0\,a_0^2b_0^2c_0^2+z_1\,a_0a_1b_0^2c_0^2+\dots+z_{26}\,a_1^2b_1^2c_1^2.
\]
We have that, for example, 
$a_0^*b_0^*c_0^*(f)=8z_0\,
a_0b_0c_0+4z_1\,a_1b_0c_0+4z_3\,
a_0b_1c_0+2z_4\,a_1b_1c_0+4z_9\,
a_0b_0c_1+2z_{10}\,a_1b_0c_1+2z_{12}\,
a_0b_1c_1+z_{13}\,a_1b_1c_1$.
Thus the equation of $S$ is given 
by the vanishing of the determinant
\[
\left| \begin {array}{cccccccc} 8\,z_{{0}}&4\,z_{{9}}&4\,z_{{3}}&2\,z
_{{12}}&4\,z_{{1}}&2\,z_{{10}}&2\,z_{{4}}&z_{{13}}
\\ \noalign{\medskip}4\,z_{{1}}&2\,z_{{10}}&2\,z_{{4}}&z_{{13}}&8\,z_{
{2}}&4\,z_{{11}}&4\,z_{{5}}&2\,z_{{14}}\\ \noalign{\medskip}4\,z_{{3}}
&2\,z_{{12}}&8\,z_{{6}}&4\,z_{{15}}&2\,z_{{4}}&z_{{13}}&4\,z_{{7}}&2\,
z_{{16}}\\ \noalign{\medskip}2\,z_{{4}}&z_{{13}}&4\,z_{{7}}&2\,z_{{16}
}&4\,z_{{5}}&2\,z_{{14}}&8\,z_{{8}}&4\,z_{{17}}\\ \noalign{\medskip}4
\,z_{{9}}&8\,z_{{18}}&2\,z_{{12}}&4\,z_{{21}}&2\,z_{{10}}&4\,z_{{19}}&
z_{{13}}&2\,z_{{22}}\\ \noalign{\medskip}2\,z_{{10}}&4\,z_{{19}}&z_{{
13}}&2\,z_{{22}}&4\,z_{{11}}&8\,z_{{20}}&2\,z_{{14}}&4\,z_{{23}}
\\ \noalign{\medskip}2\,z_{{12}}&4\,z_{{21}}&4\,z_{{15}}&8\,z_{{24}}&z
_{{13}}&2\,z_{{22}}&2\,z_{{16}}&4\,z_{{25}}\\ \noalign{\medskip}z_{{13
}}&2\,z_{{22}}&2\,z_{{16}}&4\,z_{{25}}&2\,z_{{14}}&4\,z_{{23}}&4\,z_{{
17}}&8\,z_{{26}}\end{array} \right| = 0.
\]

\section*{Acknowledgements}
It is a pleasure to thank Giorgio Ottaviani
for suggesting us the idea of how to calculate
the equation of the secant variety of the 
$(2,2,2)$-embedding of $(\pp^1)^3$.
We are pleased as well to thank Chiara 
Brambilla for many useful conversations.

\end{document}